\documentclass[11pt]{article}
\usepackage{lineno,hyperref}
\usepackage{bbm}
\usepackage{enumitem}
\usepackage{graphicx}
\usepackage{amsmath}
\usepackage{mismath}
\usepackage{eurosym}
\usepackage{float}
\usepackage{xcolor}
\usepackage{mathrsfs}
\usepackage{graphicx}
\usepackage[utf8]{inputenc}
\usepackage{amsmath}
\usepackage{amstext}
\usepackage{amsfonts}
\usepackage{amsthm}
\usepackage{amssymb}
\usepackage[bf,SL,BF]{subfigure}
\usepackage{subfigure}
\usepackage{caption}
\newcommand{\triple}{{\vert\kern-0.25ex\vert\kern-0.25ex\vert}}
\allowdisplaybreaks
\newtheorem{theorem}{Theorem}

\newtheorem{corollary}[theorem]{Corollary}

\newtheorem{example}[theorem]{Example}

\newtheorem{lemma}[theorem]{Lemma}

\newtheorem{proposition}[theorem]{Proposition}
\newtheorem{remark}[theorem]{Remark}

\input{tcilatex}
\begin{document}

\title{On the Analysis of a Singular Stochastic Volterra Differential Equation driven by a Wiener Noise}
\author
{{\bf Emmanuel Coffie\footnote{University of Liverpool, Liverpool, L69 7ZL, U.K; emmanuel.coffie@liverpool.ac.uk}
\quad \bf Olivier Menoukeu-Pamen\footnote{University of Liverpool, Liverpool,L69 7ZL, U.K; o.menoukeu-pamen@liverpool.ac.uk}
 \quad \bf Frank Proske\footnote{University of Oslo, 0316 Oslo, Norway; proske@math.uio.no}}}
\date{}
\maketitle
\begin{abstract}

In this article, we construct unique strong solutions to a  class of stochastic Volterra differential equations driven by a singular drift vector field and a Wiener noise. Further, we examine the Sobolev differentiability of the strong solution with respect to its initial value.

\bigskip

\noindent \emph{keywords}: Singular Volterra SDEs, Brownian motion, Malliavin calculus

\end{abstract}

\section{Introduction}

Stochastic Volterra differential equations (SDVs) were first studied by Berger and Mizel (\cite{BM80a}, \cite{BM80b}) and have since been successfully applied to modeling the dynamics of various phenomena, including population growth, the spread of epidemics, and (rough) stock price volatility in finance (see, e.g., \cite{AJEE19b},\cite{BFG16}). In contrast to classical stochastic differential equations, SVDEs incorporate time‑dependent kernels which introduce long‑range dependence. These equations are particularly effective in representing dynamical systems that exhibit non-Markovian behavior.
\par
In this article, we study SVDEs with spatially singular and merely measurable drift coefficients, where classical Lipschitz or Hölder assumptions fail. In particular, we study the following stochastic Volterra
differential equation (SVDE)%
\begin{equation}
X_{t}^{x}=x+\int_{0}^{t}b(t,s,X_{s}^{x})ds+B_{t}, \label{SDE0}
\end{equation}
$0\leq t\leq T$, $x\in \mathbb{R}^{d}$, where $b:\left[ 0,T\right] \times \left[ 0,T\right] \times \mathbb{R}^{d}\longrightarrow \mathbb{R}^{d}$ is a Borel-measurable function and $(B_{t})_{t\in [0,T]}$ is a d-dimensional Brownian motion (on a filtered probability space).
Our primary objective is to construct unique strong solutions to the stochastic Volterra differential equation (SVDE) \eqref{SDE0}, where the driving vector field 
$b$ is spatially singular, that is, it is not Lipschitz continuous and may have discontinuities with respect to the variable $x$ (and $s$). For this analysis, it is assumed that the driving vector field $b$ in \eqref{SDE0} admits a Volterra‑type expansion:
\begin{equation}\label{b}
b(t,s,x)=\sum_{m\geq 0}(t-s)^{m}g_{m}(s,x),  
\end{equation}%
$0\leq s\leq t\leq T$, $x\in \mathbb{R}^{d}$, where $g_{m}\in L^{\infty }(\left[ 0,T\right] \times \mathbb{R}^{d};\mathbb{R%
}^{d})$, such that for all $m\ge 0$ 
\begin{equation}\label{gm}
\big \vert  g_m\big\vert_{\infty}^{1-\epsilon}\le C\frac{m^{(\frac{1}{2}-\epsilon)m}}{m!}
\end{equation}
for some $\epsilon \in (0,\frac{1}{2})$ and some constant $C<\infty$.
\par
In this paper, our objective is to analyze the SVDE \eqref{SDE0} for the class $\mathbb{H}$ of vector fields $b\in L^{\infty }(\left[ 0,T\right] ^{2}\times 
\mathbb{R}^{d};\mathbb{R}^{d})$ of the form (\ref{b}). In fact, we prove
for vector fields $b\in \mathbb{H}$ that the SVDE (\ref{SDE0}) has a unique
strong solution which is Malliavin differentiable and ($\mathbb{P}-$a.e.) locally
Sobolev differentiable with respect to the initial condition (see Theorem %
\ref{Main}). Our approach for the construction of strong solutions is based on techniques of Malliavin calculus which was developed in \cite{MMNPZ} and \cite{MBP} in the context of singular stochastic differential equations with an additive Wiener noise.
\par
We mention that in the case of a non-singular kernel the authors in \cite%
{AJEE19b} and \cite{PS23b} obtain a unique strong solution for $\frac{1}{2}$%
-H\"{o}lder continuous diffusion coefficients on the real line by employing e.g. the theorem of Yamada-Watanabe. On the other hand, the case of singular
kernels (as e.g. of the type $(t-s)^{-\alpha }$) and spatially Lipschitz
continuous vector fields was investigated in \cite{CLP95} and \cite{CD01}.
Furthermore, based on PDE techniques the authors in \cite{MS15} establish
pathwise uniqueness of one-dimensional SVDEs with singular kernels and $\xi $%
-H\"{o}lder continuous diffusion coefficients and drift coefficients not
depending on the solution. See also \cite{PS23} for an extension of the
latter result to the case of time-inhomogeneous coefficients. As for the concept of weak solutions of singular SVDEs we refer to \cite{AJ21}, \cite{MS15} and the references therein.  
\par
While the Malliavin framework is inspired by earlier work in \cite{MMNPZ,MBP}, the Volterra framework
creates new analytic difficulties that are absent in the Markovian SDE case:
\begin{itemize}[label={-}]
  \item the equation is non-Markovian and lacks the flow/semigroup structure that underlies many classical arguments;
  \item the drift involves time-dependent memory via $\int_0^t b(t,s,\cdot)\,ds$, which changes the structure of
  Picard iterations and derivative representations;
  \item new kernel-dependent estimates are required to control iterated integrals and ensure convergence on small
  time intervals;
  \item the proof of strong well-posedness under merely measurable drifts in this non-Markovian setting requires a careful combination of approximation, weak convergence, and identification of adapted limits.
\end{itemize} 
Thus the present paper constitutes a genuine extension rather than a small technical adaptation and therefore:
\begin{itemize}[label={-}]
\item establishes strong existence and uniqueness for non‑Markovian SVDEs with merely measurable drift.
\item proves a regularization by noise phenomenon in the Volterra setting, previously unknown.
\item derives Malliavin differentiability and Sobolev regularity in the initial condition.
\end{itemize}
In other words, we aim to establish strong well-posedness and regularity for a new class of singular stochastic Volterra equations under minimal assumptions. The results demonstrate a robust regularization-by-noise phenomenon in a non-Markovian framework. The Volterra structure fundamentally alters the analysis since classical semimartingale and Markov techniques are no longer applicable.
\par
Finally, let us mention that Volterra equations are a canonical way to model non-Markovian dynamics with memory and that, singular kernels appear naturally in rough and long-memory modelling. See, e.g., \cite{AJEE19b,BFG16} and the references therein. Further, our focus is on a basic, robust regularization-by-noise phenomenon: additive Brownian noise restores well-posedness even when the drift is merely measurable (in the spatial variable). In this case, we choose the kernel class $\mathbb{H}$ to allow a transparent analytic route to strong solutions and regularity with respect to the initial value; it is not meant as a complete coverage of all rough-volatility kernels, but it provides a rigorous first step for singular SVDEs under minimal drift regularity.
\par
The remainder of the paper is organised as follows: We prove that SVDE \eqref{SDE0} has a unique strong solution in Section \ref{global}. In Section \ref{convergence}, we establish the weak convergence of approximating solutions. We then discuss the existence of a unique weak solution in Section \ref{existence1} and provide an example to support the main result.

\section{Strong uniqueness}\label{global}
We are coming to the main result of our paper:
\begin{theorem}
\label{Main}
Suppose that $b\in $ $\mathbb{H}$. Then there exists for a sufficiently small time horizon $T$ a unique strong
solution $X_{t}^{x},0\leq t\leq T$ to the \textup{SVDE} \eqref{SDE0} for all initial
values $x$. Moreover, $X_{t}^{x}$ is Malliavin differentiable and the
associated stochastic flow $X_{t}^{\cdot }$ belongs to the space $%
L^{2}(\Omega ;W_{loc}^{1,2}(\mathbb{R}^{d}))$ for all $0\leq t\leq T$.
\end{theorem}
We also need the following auxiliary result for the proof of the main result:
\begin{lemma}
\label{Compactness}
\bigskip Let $b\in \mathbb{H}$ given by
\begin{equation*}
b(t,s,x)=\sum_{\mu \geq 0}(t-s)^{\mu }\cdot b_{\mu }(s,x)
\end{equation*}
for $b_{\mu }\in C_{c}^{\infty}((0,T)\times \mathbb{R}^{d};\mathbb{R}^{d})$, $\mu \geq 0,$ such that 
\begin{equation*}
\sup_{\mu \geq 0}\vert\vert D b_{\mu }\vert\vert _{\infty }<\infty,
\end{equation*}%
where $D$ is the spatial derivative of $ b_{\mu }$. Let $X_{\cdot }^{x}$ be the unique strong solution associated with the \textup{SVDE} \eqref{SDE0}. Denote by $D_{u}$ the Malliavin derivative with respect to the
Brownian motion $B_{\cdot }$. Then for a sufficiently small $T$
\begin{equation*}
\sup_{0\leq u\leq T}\big\vert\big\vert D_{u}X_{t}^{x}\big\vert\big\vert_{L^{2}(\Omega )}\leq
C_{d,T}(\sup_{\mu \geq 0}\vert\vert  b_{\mu }\vert\vert_{\infty })
\end{equation*}
as well as there exists  $\alpha =\alpha (s)>0$ such that
\begin{equation*}
\big\vert\big\vert D_{r}X_{s}^{x}-D_{v}X_{s}^{x}\big\vert\big\vert_{L^{2}(\Omega )}\leq C_{d,T}(\sup_{\mu
\geq 0}\left\Vert b_{\mu }\right\Vert _{\infty })\left\vert r-v\right\vert
^{\alpha }
\end{equation*}%
for $0\leq v\leq r\leq T$, where $C_{d,T}:\left[
0,\infty \right) \longrightarrow $ $\left[ 0,\infty \right) $ is an
increasing, continuous function, $\left\Vert \cdot \right\Vert $ a matrix
norm and $\left\Vert \cdot \right\Vert _{\infty }$ the supremum norm.
\end{lemma}

\begin{proof}
We know by assumption that 
\begin{equation*}
b(t,s,x)=\sum_{\mu \geq 0}(t-s)^{\mu }\cdot b_{\mu }(s,x).
\end{equation*}%
So we can use dominated convergence and obtain that 
\begin{equation*}
Db(t,s,x)=\sum_{\mu \geq 0}(t-s)^{\mu }\cdot Db_{\mu }(s,x)
\end{equation*}%
pointwise.  By applying the Malliavin derivative $D$ to both sides of the SVDE 
\begin{equation*}
X_{t}^{x}=x+\int_{0}^{t}b(t,s,X_{s}^{x})ds+B_{t}\text{,}
\end{equation*}%
we obtain the equation 
\begin{equation}
D_{u}X_{t}^{x}=\int_{u}^{t}Db(t,s,X_{s}^{x})\cdot D_{u}X_{s}^{x}ds+I_{d}
\end{equation}%
fort $u\leq t$ a.e (compare \cite{Nualart}). Here, $Db$ denotes the spatial derivative of the vector
field $b:[0,T]^{2}\times \mathbb{R}^{d}\rightarrow \mathbb{R}^{d}$ and $%
I_{d}\in \mathbb{R}^{d\times d}$ the unit matrix. Using Picard
iteration, we then find 
\begin{equation*}
D_{u}X_{t}^{x}=I_d+\sum_{n\geq 2}\Big\{\int_{u}^{t}\int_{u}^{s_{n-1}}...%
\int_{u}^{s_{2}}Db(t,s_{n-1},X_{s_{n-1}}^{x})\cdot ...\cdot
Db(s_{2},s_{1},X_{s_{1}}^{x})ds_{1}...ds_{n-1}\Big\}
\end{equation*}%
in $L^{p}$, $p\geq 1$. Thus, we derive that 
\begin{align*}
D_{u}X_{t}^{x}&=I_d+\sum_{n\geq 2}\sum_{\mu_2,...,\mu_n\ge 0}\Big\{\int_{u}^{t}%
\int_{u}^{s_{n-1}}...\int_{u}^{s_{2}}(t-s_{n-1})^{\mu_n
}Db_{\mu_n }(s_{n-1},X_{s_{n-1}}^{x})\cdot ...\cdot\\
& (	s_2-s_1)^{\mu_2}Db_{\mu_2 }(s_1,X_{s_1}^{x})ds_1...ds_{n-1}\Big\}.
\end{align*}%
The latter implies that 
\begin{align*}
\big\vert\big\vert D_{u}X_{t}^{x}\big\vert\big\vert_{L^{2}(\Omega )}
&\le  d^{\frac{1}{2}}+\sum_{n\geq 2}\sum_{\mu_2,...,\mu_n\ge 0}\Big\vert\Big\vert\int_{u}^{t}%
\int_{u}^{s_{n-1}}...\int_{u}^{s_{2}}(t-s_{n-1})^{\mu_n
}Db_{\mu_n }(s_{n-1},X_{s_{n-1}}^{x})\cdot ...\cdot\\
& (	s_2-s_1)^{\mu_2}Db_{\mu_2 }(s_1,X_{s_1}^{x})ds_1...ds_{n-1}\Big\vert\Big\vert_{L^{2}(\Omega )}.
\end{align*}%
Further, it also follows from dominated convergence combined with the Cauchy
formula for repeated integration that%
\begin{eqnarray*}
X_{t}^{x} &=&x+\int_{0}^{t}b(t,s,X_{s}^{x})ds+B_{t} \\
&=&x+\sum_{\mu \geq 0}\int_{0}^{t}(t-s)^{\mu }b_{\mu }(s,X_{s}^{x})ds+B_{t}
\\
&=&x+\sum_{\mu \geq 0}\mu!\int_{0}^{t}\int_{0}^{s_{1}}...\int_{0}^{s_{\mu}}b_{\mu }(s,X_s^{x})ds ds_{\mu }...ds_1+B_{t} \\
&=&x+\int_{0}^{t}C_{s_{1}}(X_{\cdot }^{x})ds_{1}+B_{t}\text{,}
\end{eqnarray*}%
where 
\begin{equation*}
C_{s_{1}}(X_{\cdot }^{x}):=\sum_{\mu \geq 0}\mu!\int_{0}^{s_{1}}...\int_{0}^{s_{\mu}}b_{\mu }(s ,X_s^{x})dsds_{\mu }...ds_{2}\text{.}
\end{equation*}%
Then, using Girsanov's theorem applied to the process $C_{t}(X_{\cdot
}^{x}),0\leq t\leq T$, in combination with the proof of Lemma 3.5 in \cite%
{MMNPZ}, we obtain that 
\begin{align*}
&
\Big\vert\Big\vert\int_{u}^{t}%
\int_{u}^{s_{n-1}}...\int_{u}^{s_{2}}(t-s_{n-1})^{\mu_n
}Db_{\mu_n }(s_{n-1},X_{s_{n-1}}^{x})\cdot ...\cdot (	s_2-s_1)^{\mu_2}Db_{\mu_2 }(s_1,X_{s_1}^{x})ds_1...ds_{n-1}\Big\vert\Big\vert_{L^{2}(\Omega )}\\
& \leq C_{T,b}\cdot
\Big\vert\Big\vert\int_{u}^{t}%
\int_{u}^{s_{n-1}}...\int_{u}^{s_{2}}(t-s_{n-1})^{\mu_n
}Db_{\mu_n }(s_{n-1},B_{s_{n-1}}^{x})\cdot ...\cdot (	s_2-s_1)^{\mu_2}Db_{\mu_2 }(s_1,B_{s_1}^{x})ds_1...ds_{n-1}\Big\vert\Big\vert_{L^{8}(\Omega )},
\end{align*}%
for a constant $C_{T,b}$ depending on $T$ and $\sup_{\mu \geq 0}\left\Vert
b_{\mu }\right\Vert _{\infty }$, where $B_{t}^{x}:=x+B_{t}$. On the other hand, we see from Cauchy`s formula for repeated integration that%
\begin{eqnarray*}
&&\int_{u}^{t}%
\int_{u}^{s_{n-1}}...\int_{u}^{s_{2}}(t-s_{n-1})^{\mu_n
}Db_{\mu_n }(s_{n-1},B_{s_{n-1}}^{x})\cdot ...\cdot (	s_2-s_1)^{\mu_2}Db_{\mu_2 }(s_1,B_{s_1}^{x})ds_1...ds_{n-1}\\
&&=\prod_{i=2}^{n}(\mu_i)!\int_{u}^{t}\int_{u}^{s_{n,\mu_n}}...\int_{u}^{s_{n,1}}...\int_{u}^{s_2}\int_{u}^{s_{2,\mu_2}}...\int_{u}^{s_{2,1}}Db_{\mu_n }(s_{n-1},B_{s_{n-1}}^{x})\cdot ...\cdot Db_{\mu_2 }(s_1,B_{s_1}^{x})\\
&&ds_1...ds_{2,1}...ds_{2,\mu_2}...ds_{n-1}ds_{n,1}...ds_{n,\mu_n}\text{.}
\end{eqnarray*}%
Finally, using integration by parts combined with Proposition \ref{prop1}
we can argue very similarly to the proof of Lemma 3.5 in \cite{MMNPZ} and
obtain the following estimate%

\begin{eqnarray*}
&&\sup_{0\leq u\leq t}\big\vert\big\vert D_{u}X_{t}^{x}\big\vert\big\vert_{L^{2}(\Omega )} \\
&\leq &C_{1}\Big(1+\sum_{n\geq 2}\sum_{\mu_2,...,\mu_n\ge 0}\frac{\prod_{i=2}^{n}(\mu_i)! d^{\sum_{i=2}^n(\mu_i +1)+2}2^{4\sum_{i=2}^n(\mu_i +1)}C^{\sum_{i=2}^n(\mu_i+1)}\prod_{i=2}^n\left\Vert b_{\mu_i
}\right\Vert _{\infty }T^{\frac{1}{2}\sum_{i=2}^n(\mu_i+1)}}{\Gamma (4\sum_{i=2}^n(\mu_i+1)+1)^{1/8}}\Big) \\
&\leq &C_{d,T}(\sup_{\mu \geq 0}\left\Vert b_{\mu }\right\Vert _{\infty })%
\text{,}
\end{eqnarray*}%
where $C_{d,T}$ is a function as stated in the Lemma. Using Stirling`s formula, it follows from our assumptions on $\left\Vert b_{\mu }\right\Vert _{\infty }$, $\mu\ge 0$, that the above double sum converges, provided the time horizon $T$ is chosen small enough. In a very similar way, we also get the second estimate in the Lemma.
\end{proof}

\begin{remark}
The power-series assumption is used for two reasons:
\begin{itemize}[label={-}]
\item it yields a repeated-integration representation (Cauchy formula) that converts Volterra drifts into a form amenable to Malliavin estimates and compactness arguments; and
\item it provides uniform control of iterated integral terms via \eqref{gm},
  which is essential for summability in the derivative bounds.
\end{itemize}
Thus the assumption is tailored to guarantee convergence of the expansion on short time horizons and does not aim to cover all kernels used in applications.
\end{remark}

\bigskip

\section{Weak convergence of approximating solutions}\label{convergence}
The following technical lemma deals with weak convergence of approximating solutions to weak solutions and will be used in the proof of the main result (Theorem \ref{Main}) in Section \ref{4} for establishing a "transformation property" of strong solutions (see \eqref{transform}).

\begin{lemma}\label{WeakConvergence} 
Let $T<1$. Consider the processes $C_{\cdot }$ and $C_{\cdot }^{(n)},n\geq 1$ given by
\begin{equation*}
C_{s_{1}}^{(n)}(B_{\cdot }^{x}):=\sum_{\mu \geq 0}\mu!\int_{0}^{s_{1}}...\int_{0}^{s_{\mu}}b_{\mu }^{(n)}(s ,B_s^{x})dsds_{\mu }...ds_{2}
\end{equation*}%
for vector fields $b_{\mu }^{(n)}\in C_{c}^{\infty }(\left( 0,T\right)
\times \mathbb{R}^{d};\mathbb{R}^{d}),n\geq 1$ and 
\begin{equation*}
C_{s_{1}}(B_{\cdot }^{x}):=\sum_{\mu \geq 0}\mu!\int_{0}^{s_{1}}...\int_{0}^{s_{\mu}}b_{\mu }(s ,B_s^{x})dsds_{\mu }...ds_{2},
\end{equation*}%
where 
\begin{equation*}
b_{\mu }\in L^{\infty }(\left[ 0,T\right] \times \mathbb{R}^{d};\mathbb{R}%
^{d}), \mu\ge 0\text{.}
\end{equation*}%
Assume that%
\begin{equation*}
b_{\mu }^{(n)}(t,y)\underset{n\longrightarrow \infty }{\longrightarrow }%
b_{\mu }(t,y)\text{ }(t,y)-a.e.
\end{equation*}%
as well as 
\begin{equation*}
\sup_{n\geq 1,\mu\ge 0}\big\vert\big\vert b_{\mu }^{(n)}\big\vert\big\vert_{L^{\infty }(\left[ 0,T%
\right] \times \mathbb{R}^{d};\mathbb{R}^{d})}<\infty .
\end{equation*}%
Let $X_{.}^{x,n}$ be the strong solution of the \textup{SVDE} associated with the
$(b_{\mu}^{(n)})_{\mu\ge 0, n\geq 1}$. Suppose that $X_{.}^{x}$ is a weak
solution of the \textup{SVDE} with respect to $(b_{\mu })_{\mu\ge 0}$ on the same probability
space. Then for all bounded continuous functions 
\begin{equation*}
\varphi (X_{t}^{x,n})\underset{n\rightarrow \infty }{\longrightarrow }%
E[\varphi (X_{t}^{x})|\mathcal{F}_{t}]
\end{equation*}%
weakly in $L^{2}(\Omega ;\mathcal{F}_{t})$.
\end{lemma}

\begin{proof}
Using dominated convergence, one shows for all $t$ that
\begin{equation*}
C_{t}^{(n)}(B_{\cdot }^{x})\underset{n\rightarrow \infty }{\longrightarrow }%
C_{t}(B_{\cdot }^{x})\text{ }a.e.
\end{equation*}
Further, it follows from the It\^{o} isometry and uniform integrability that 
\begin{equation*}
\int_{0}^{T}C^{(n)}(s,B_{\cdot }^{x})dB_{s}\underset{n\longrightarrow \infty }{%
\longrightarrow }\int_{0}^{T}C(s,B_{\cdot }^{x})dB_{s}
\end{equation*}%
in $L^{p}(\Omega )$ for all $p\geq 1$. We note that \newline
\newline
$\prod_{t}:=\Big\{\exp
(\sum_{j=1}^{d}\int_{0}^{t}f_{i}(s)dB_{s}^{j}):,f_{j}\in L^{\infty
}([0,t]),j=1,2,..,d\Big\}$ \newline
\newline
is a total subset of $L^{p}(\Omega ;\mathcal{F}_{t})$. So it is sufficient
to show that 
\begin{equation*}
E\left[ \varphi (X_{t}^{x,n})\mathcal{E}(f)\right] \underset{n\rightarrow
\infty }{\longrightarrow }E[E[\varphi (X_{t}^{x})|\mathcal{F}_{t}]\mathcal{E}%
(f)]
\end{equation*}%
for all $\mathcal{E}(f):=\exp
(\sum_{j=1}^{d}\int_{0}^{t}f_{i}(s)dB_{s}^{j})\in \prod_{t}$. Define the
Girsanov change of measures 
\begin{equation*}
dQ_{n}=\mathcal{E}(\int_{0}^{\cdot }\langle
C^{(n)}(s,X_{\cdot }^{x,n}),dB_{s}\rangle )d\mathbb{P},\quad n\geq 1.
\end{equation*}%
Then the Girsanov's theorem implies that 
\begin{equation*}
B_{.}^{\ast ,n}:=B_{t}-\int_{0}^{.}C^{(n)}(s,X_{\cdot }^{x,n})ds
\end{equation*}%
is a $Q_{n}$-Brownian motion and by using the notation $C^{(n)}(s,\cdot )=$ $%
(C^{(n,1)}(s,\cdot ),...,C^{(n,d)}(s,\cdot ))^{\text{T}}$ ($\text{T}$ transpose) we
get   
\begin{align*}
& E\Big[\varphi (X_{t}^{x,n})\exp \Big\{\sum_{j=1}^{d}%
\int_{0}^{t}f_{j}(s)dB_{s}^{j}\Big\}\Big] \\
& =E_{Q_{n}}\Big[\varphi (x+B_{t}^{\ast ,n})\exp \Big\{\sum_{j=1}^{d}%
\int_{0}^{t}f_{j}(s)dB_{s}^{\ast ,n,j} \\
& +\sum_{j=1}^{d}\int_{0}^{t}f_{j}(s)C^{(n,j)}(s,x+B_{\cdot }^{\ast ,n})ds\Big\}%
\times \mathcal{E}(\int_{0}^{\cdot }\langle C^{(n)}(s,x+B_{\cdot }^{\ast
,n}),dB_{s}^{\ast ,n}\rangle) \Big] \\
& =E\Big[\varphi (x+B_{t})\exp \Big\{\sum_{j=1}^{d}%
\int_{0}^{t}f_{j}(s)dB_{s}^{j}+\sum_{j=1}^{d}%
\int_{0}^{t}f_{j}(s)C^{(n,j)}(s,x+B_{\cdot })ds\Big\} \\
& \times \mathcal{E}(\int_{0}^{\cdot }\langle
C^{(n)}(s,x+B_{\cdot }),dB_{s}\rangle )\Big].
\end{align*}%
We observe that
\begin{align*}
& E\Big[E\Big[\varphi (X_{t}^{x})\Big\vert\mathcal{F}_{t}\Big]\exp \Big\{%
\sum_{j=1}^{d}\int_{0}^{t}f_{j}(s)dB_{s}^{j}\Big\}\Big] \\
& =E\Big[E\Big[\varphi (X_{t}^{x})\exp \Big\{\sum_{j=1}^{d}%
\int_{0}^{t}f_{j}(s)dB_{s}^{j}\Big\}\Big\vert\mathcal{F}_{t}\Big]\Big] \\
& =E\Big[\varphi (X_{t}^{x})\exp \Big\{\sum_{j=1}^{d}%
\int_{0}^{t}f_{j}(s)dB_{s}^{j}\Big\}\Big].
\end{align*}%
Using the inequality $|e^{x}-e^{y}|\leq e^{x+y}|x-y|$ for all $x,y$, our
assumptions and the H\"{o}lder inequality, we derive that 
\begin{align*}
& \Big\vert E\Big[\varphi (X_{t}^{x,n})\exp \Big\{\sum_{j=1}^{d}%
\int_{0}^{t}f_{j}(s)dB_{s}^{j}\Big\}\Big]-E\Big[\varphi (X_{t}^{x})\exp %
\Big\{\sum_{j=1}^{d}\int_{0}^{t}f_{j}(s)dB_{s}^{j}\Big\}\Big]\Big\vert \\
& \leq \Big\vert E\Big[\varphi (x+B_{t})\exp \Big\{\sum_{j=1}^{d}%
\int_{0}^{t}f_{j}(s)dB_{s}^{j}+\sum_{j=1}^{d}%
\int_{0}^{u}f_{j}(s)C^{(n,j)}(s,x+B_{\cdot })ds\Big\} \\
& \times \mathcal{E}(\int_{0}^{\cdot }\langle
C^{(n)}(s,x+B_{\cdot }),dB_{s}\rangle) \Big] \\
& -E\Big[\varphi (x+B_{t})\exp \Big\{\sum_{j=1}^{d}%
\int_{0}^{t}f_{j}(s)dB_{s}^{j}+\sum_{j=1}^{d}%
\int_{0}^{t}f_{j}(s)C^{(j)}(s,x+B_{\cdot })ds\Big\} \\
& \times \mathcal{E}(\int_{0}^{\cdot }\langle C(s,x+B_{\cdot }),dB_{s}\rangle) %
\Big] \Big\vert\\
& \leq K_{t}(f)J_{1}^{n}J_{2}^{n},
\end{align*}%
where 
\begin{equation*}
K_{t}(f):=E\Big[\varphi (x+B_{t})^{2}\exp \Big\{2\sum_{j=1}^{d}%
\int_{0}^{t}f_{j}(s)dB_{s}^{j}\Big\}\Big]^{1/2},
\end{equation*}

\begin{align*}
J_{1}^{n}& :=E\Big[\exp \Big\{4\sum_{j=1}^{d}%
\int_{0}^{t}f_{j}(s)(C^{(n,j)}(s,x+B_{\cdot })+C^{(j)}(s,x+B_{\cdot }))ds\Big\} \\
& \times \exp \Big\{4\int_{0}^{T}\langle
C^{(n)}(s,x+B_{\cdot })+C(s,x+B_{\cdot }),dB_{s}\rangle \Big\} \\
& \times \exp \Big\{-2\int_{0}^{T}\Big(%
\big\vert\big\vert C^{(n)}(s,x+B_{\cdot })\big\vert\big\vert^{2}+\big\vert\big\vert C(s,x+B_{\cdot })\big\vert\big\vert^{2}\Big)ds\Big\}\Big]%
^{1/4}
\end{align*}%
\begin{align*}
J_{2}^{n}& :=E\Big[\Big\vert\sum_{j=1}^{d}%
\int_{0}^{t}f_{j}(s)C^{(j)}(s,x+B_{\cdot })ds+\sum_{j=1}^{d}%
\int_{0}^{T}C^{(j)}(s,x+B_{\cdot })dB_{s}^{j} \\
& -\frac{1}{2}\int_{0}^{T}\big\vert\big\vert C(s,x+B_{\cdot })\big\vert\big\vert^{2}ds-\sum_{j=1}^{d}%
\int_{0}^{t}f_{j}(s)C^{(n,j)}(s,x+B_{\cdot })ds \\
& -\sum_{j=1}^{d}\int_{0}^{T}C^{(n,j)}(s,x+B_{\cdot })dB_{s}^{j}+\frac{1}{2}%
\int_{0}^{T}\Big(\big\vert\big\vert C^{(n)}(s,x+B_{\cdot})\big\vert\big\vert^{2}ds\Big\vert^{2}\Big]^{1/4}.
\end{align*}%
Then it follows from the martingale property of the Doleans-Dade exponential
and our assumptions on the coefficients that $\sup_{n\geq 1}J_{1}^{n}<\infty 
$. Further, using once again dominated convergence, we have that $J_{2}^{n}%
\underset{n\rightarrow \infty }{\longrightarrow }0$, which yields the result.
\end{proof}

\begin{remark}\label{Identity}
\begin{enumerate}[label=(\roman*)]
\item The proof of the above auxiliary result shows that the function $\varphi $ can also be chosen to be the identity map. 
\item Using mollification (with a non-negative mollifier), one can choose a sequence $b^{(k)}$, $k\ge 1$, of finite sums of the type \eqref{b} such that the conditions of Lemma \ref{Compactness} and Lemma \ref{WeakConvergence} are satisfied.
\end{enumerate}
\end{remark}

\section{Existence of a unique weak solution and proof of the main result}\label{existence1}\label{4}
\begin{lemma}\label{existence}
Let $B_{t},0\leq t\leq T$ be a $d-$dimensional Brownian motion on a
probability space $(\Omega ,A,\mathbb{P})$. Suppose $b\in \mathbb{H}$. Then there exists a weak solution $X_{.}$ to \textup{SVDE} \eqref{SDE0},
which is unique in law.
\end{lemma}

\begin{proof}
Define the probability measure 
\begin{equation}
dQ=\mathcal{E}(\int_{0}^{\cdot }\langle C_{s}(B_{\cdot }^{x}),dB_{s}\rangle
)_{T}d\mathbb{P}.  \label{weak}
\end{equation}%
Then by the Girsanov theorem, we have 
\begin{equation*}
B_{t}^{\ast }=B_{t}-\int_{0}^{t}C_{s}(B_{\cdot }^{x})dB_{s},0\leq t\leq T
\end{equation*}%
as a $Q-$Wiener process. Then a weak solution is obtained by choosing the
process $X_{t}:=B_t^x$. To prove uniqueness in law of weak solutions to 
\textup{SVDE} \eqref{SDE0}, let $X_{.}$ and $Y_{.}$ be two weak solutions
with respect to the measures $\mathbb{P}$ and $\mathbb{P}^{\ast }$ and
Brownian motions $B_{\cdot }^{X}$ and $B_{\cdot }^{Y}$, repectively. Then by
means of Girsanov`s theorem and the replacement of $B_{\cdot }^{X}$ and $%
B_{\cdot }^{Y}$ by $B_{\cdot }$ under $\mathbb{P}$ in connection with
measurable functionals on the Wiener space, we get for all Borel sets $A$ in 
$\mathbb{R}^{n}$ that 
\begin{equation*}
\mathbb{P}((X_{t_{1}},...,X_{t_{n}})\in A)=E_{\mathbb{P}}\Big[%
1_{\{(B_{t_{1}}^x,...,B_{t_{n}}^x)\in A\}}\mathcal{E}(\int_{0}^{\cdot }\langle
C_{s}(B_{\cdot }^{x}),dB_{s}\rangle )_{T}\Big].
\end{equation*}%
as well as 
\begin{equation*}
\mathbb{P}^{\ast }((Y_{t_{1}},...,Y_{t_{n}})\in A)=E_{\mathbb{P}}\Big[%
1_{\{(B_{t_{1}}^x,...,B_{t_{n}}^x)\in A\}}\mathcal{E}(\int_{0}^{\cdot }\langle
C_{s}(B_{\cdot }^{x}),dB_{s}\rangle )_{T}\Big].
\end{equation*}%
So $X_{.}$ and $Y_{.}$ coincide in law.
\end{proof}

We are now coming to the proof of our main result.

\begin{proof}[Proof of Theorem \protect\ref{Main}]
\bigskip
Let $X_{\cdot }^{x,n}$ be the unique strong solution to
the SVDE \eqref{SDE0} associated with coefficients $b_{\mu }^{(n)},\mu \geq
0,n\geq 1$ converging to $b_{\mu }\in L^{\infty }(\left[ 0,T\right] \times 
\mathbb{R}^{d};\mathbb{R}^{d})$ for $n\longrightarrow \infty $ for all $\mu $
under the assumptions of Lemma \ref{WeakConvergence}, Lemma \ref{Compactness}
and     
\begin{equation*}
\sup_{n\geq 1,\mu \geq 0}\big\vert\big\vert b_{\mu }^{(n)}\big\vert\big\vert_{\infty
}<\infty \text{.}
\end{equation*}%
Then it follows from Lemma \ref{Compactness} in connection with Corollary \ref%
{MC} in the Appendix (applied to $b_{\mu }^{(n)},\mu \geq 0,n\geq 1$) that
for all $0\leq t\leq T$ there exists a subsequence $n_{k},k\geq 1$
(depending on $t$) such that%
\begin{equation*}
X_{t}^{x,n_{k}}\underset{k\longrightarrow \infty }{\longrightarrow }Y_{t}^{x}
\end{equation*}%
in $L^{2}(\Omega ;\mathcal{F}_{t})$. On the other hand, by using the theorem
of Mitoma, we can see just as in \cite{MMNPZ} that%
\begin{equation*}
X_{\cdot }^{x,n}\underset{n\longrightarrow \infty }{\longrightarrow }%
Y_{\cdot }^{x}
\end{equation*}%
in $C(\left[ 0,T\right] ;\left( S\right) ^{\ast })$, where $\left( S\right)
^{\ast }$ is the Hida-distribution space with respect to the ($d-$%
dimensional) Brownian motion $B_{\cdot }$. The latter entails that 
\begin{equation*}
X_{t}^{x,n}\underset{n\longrightarrow \infty }{\longrightarrow }Y_{t}^{x}
\end{equation*}%
in $L^{2}(\Omega ;\mathcal{F}_{t})$ for all $t$. Further, taking into
account Remark \ref{Identity}, we find that $Y_{t}^{x}$ coincides with $%
E[X_{t}^{x}|\mathcal{F}_{t}]$ a.e. for all $t$ for the weak solution $%
X_{t}^{x}$ to the SVDE \eqref{SDE0}. So for a bounded continuous function $%
\varphi $ we obtain from Lemma \ref{WeakConvergence} the following relation%
\begin{equation}\label{transform}
\varphi (E[X_{t}^{x}|\mathcal{F}_{t}])=E[\varphi (X_{t}^{x})|\mathcal{F}_{t}]%
\text{ }a.e.
\end{equation}%
for all $t$. Then approximation of $\varphi (x)=(x^{2},...,x^{2})$ by bounded continuous
functions combined with dominated convergence shows the validity of the
latter relation in this case, which implies%
\begin{equation*}
E[X_{t}^{x}|\mathcal{F}_{t}]=X_{t}^{x}\text{,}
\end{equation*}%
that is adaptedness of $X_{\cdot }^{x}$. So $X_{\cdot }^{x}$ is a strong
solution on the probability space on which the weak solution lives.
Uniqueness of strong solutions follows from the fact that the $S-$transform
of two solutions are the same (see \cite{MMNPZ}). So we get a unique strong
solution on the probability space on which the (a priori) weak solution
exists. The constructed strong solution can then be represented as a
progressively measurable functional of the Brownian motion on the latter
probability space. In order to obtain strong solutions on an arbitrary
probability space, one can replace that Brownian motion with the one on the
other probability space in the measurable functional.

Let us now show the ($\mathbb{P}-$a.e.) local Sobolev differentiability of $%
(x\longmapsto X_{t}^{x})$ . To this end, consider once more the above
approximating sequence of solutions $X_{\cdot }^{x,n},n\geq 1$. Using
dominated convergence, one gets similarly to the case of the Malliavin
derivative of a solution the linear equation
\begin{equation}
\frac{d}{dx}X_{t}^{x,n}=\int_{0}^{t}Db^{(n)}(t,s,X_{s}^{x,n})\frac{d}{dx}%
X_{s}^{x,n}ds+I_{d}\text{.}
\end{equation}%
We can then apply the same arguments as in the proof of Lemma \ref%
{Compactness} in connection with Picard iteration of the latter equation and
derive the analogous estimate%
\begin{equation*}
\sup_{0\leq t\leq T}\Big\vert\Big\vert\frac{d}{dx}X_{t}^{x,n}\Big\vert\Big\vert_{L^{2}(\Omega ;\mathbb{R}%
^{d\times d})}\leq C_{d,T}(\sup_{n\geq 1,\mu \geq 0}\big\vert\big\vert b_{\mu
}^{(n)}\big\vert\big\vert _{\infty })<\infty 
\end{equation*}%
for a non-decreasing continuous function $C_{d,T}:\left[ 0,\infty \right)
\longrightarrow $ $\left[ 0,\infty \right) $. Using the latter estimate, one
concludes (just as in \cite{MMNPZ}) that $(x\longmapsto X_{t}^{x})$ is in $%
L^{2}(W^{1,2}(U))$ for all bounded and open sets $U\subset \mathbb{R}^{d}$.
\end{proof}

\begin{example}\label{exam1}
\bigskip Consider the \textup{SVDE} 
\begin{equation}\label{Sine}
X_{t}^{x}=x+\int_{0}^{t}\sin (t-s)g(s,X_{s}^{x})ds+B_{t}\text{, } 
\end{equation}%
$0\leq t\leq T$, where $g\in L^{\infty }(\left[ 0,T\right] \times \mathbb{R}^{d};\mathbb{R}^{d})$. The drift vector field belongs to the class $\mathbb{H}$. So according to Theorem \ref{Main} the \textup{SVDE} (\ref{Sine}) posesses unique strong solution which is smooth in the sense of Malliavin differentiability and in the sense of ($\mathbb{P}-$a.e.) local Sobolev differentiability. Let us also mention here that the kernel $\sin(t-s)$ admits a finite-dimensional Markovian lift. More precisely, introducing suitable auxiliary processes, the dynamics can be expressed as
\begin{align*}
\begin{cases}
dX_t = Y_t \, dt + dW_t, \\
dY_t = \bigl(b(t,X_t) - Z_t\bigr)\, dt, \\
dZ_t = Y_t \, dt,
\end{cases}
\end{align*}
with initial conditions $(X_0,Y_0,Z_0) = (x,0,0)$.
\end{example}

\begin{example}\label{exam2}
Assume the SVDE%
\[
X_{t}=x+\int_{0}^{t}J_{0}(t-s)g(s,X_{s})ds+B_{t},0\leq t\leq T,
\]%
where $g\in L^{\infty }(\left[ 0,T\right] \times \mathbb{R};\mathbb{R})$ and 
$J_{0}$ denotes the Bessel function of the first kind of order zero, defined
by%
\[
J_{0}(u)=\sum_{m\geq 0}\frac{(-1)^{m}}{(m!)^{2}}(\frac{u}{2})^{2m}.
\]%
The kernel $K=J_{0}$ is entire and admits a power series expansion with
super-factorially decaying coefficients. Hence, the function $b$ given by 
\[
b(t,s,x)=J_{0}(t-s)g(s,x)
\]%
belongs to the class $\mathbb{H}$, and all assumptions of Theorem \ref{Main} are satisfied. In contrast to the case $K(t)=\sin (t)$ in the
previous example, this equation does not allow for a finite-dimensional
Markovian lift of the form%
\[
\begin{array}{c}
dX_{t}=c^{T}Y_{t}dt+dB_{t}, \\ 
dY_{t}=AY_{t}+bg(t,X_{t})dt,Y_{0}=0%
\end{array}%
\]%
for $A\in \mathbb{R}^{n\times n}$ and $b,c\in \mathbb{R}^{n}$. Indeed, if
such a representation were possible, then by variation of constants we would
obtain%
\[
Y_{t}=\int_{0}^{t}e^{(t-s)A}bg(s,X_{s})ds\text{,}
\]%
and hence%
\[
dX_{t}=\left( \int_{0}^{t}c^{T}e^{(t-s)A}bg(s,X_{s})ds\right) dt+dB_{t}.
\]%
Using Fubini`s theorem, this yields a Volterra equation with kernel%
\[
\widetilde{K}(u)=\int_{0}^{u}c^{T}e^{vA}bdv,u\geq 0\text{,}
\]%
and therefore its derivative is $\widetilde{K}^{\shortmid }(v)=c^{T}e^{vA}b$%
. If $\widetilde{K}(v)=J_{0}(v)$, then necessarily%
\[
\widetilde{K}^{\shortmid }(v)=J_{0}^{\shortmid }(v)=-J_{1}(v)\text{,}
\]%
where $J_{1}$ denotes the Bessel function of the first kind of order one.
However, any function of the form $u\longmapsto c^{T}e^{vA}b$ belongs to the
class of exponential-polynomial functions and solves a linear ODE with
constant coefficients. However, $J_{1}$ does not belong to this class and
solves the Bessel equation%
\[
u^{2}f^{\shortmid \shortmid }+uf^{\shortmid }+(u^{2}-1)f=0
\]%
with variable coefficients. This shows that there is no finite-dimensional
Markovian lift of the form above for $K=$ $J_{0}$.
\end{example}

\bigskip 

\begin{remark}
In the convolution-type case of $b$ with%
\[
b(t,s,x)=K(t-s)g(s,x),0\leq s\leq t\leq T,
\]%
where $g\in L^{\infty }(\left[ 0,T\right] \times \mathbb{R}^{d};\mathbb{R}%
^{d})$, the condition $b\in \mathbb{H}$ reduces to a condition on the kernel 
$K$. In particular, if $K$ admits a power series expansion%
\[
K(u)=\sum_{m\geq 0}a_{m}u^{m},0\leq u\leq T
\]%
with coefficients $a_{m},m\geq 0$ satisfying the growth condition \eqref{gm}, then $b\in \mathbb{H}$. This shows that, in the convolution-type
setting, the class $\mathbb{H}$ contains kernels with sufficiently
well-controlled power series expansions.
\end{remark}

\begin{remark}
Using the approach of this paper one may also study the \textup{SVDE} of the type%
\begin{equation*}
X_{t}^{x}=x+\int_{0}^{t}(t-s)^{\alpha }g(s,X_{s}^{x})ds+B_{t}\text{, }
\end{equation*}%
$0\leq t\leq T$, for $g\in L^{\infty }(\left[ 0,T\right] \times \mathbb{R}^{d};\mathbb{R}^{d})
$, where $-\frac{1}{2}<\alpha <0$. The reason for this is that we can write%
\begin{equation}\label{represent}
(t-s)^{\alpha }=\sum_{n\geq 0}\binom{\alpha }{n}(t-s-1)^{n}=\sum_{n\geq
0}\sum_{k=0}^{n}\binom{\alpha }{n}\binom{n}{k}(-1)^{n-k}(t-s)^{k}\text{, }%
0<t-s<2\text{,}
\end{equation}%
which allows the application of the Cauchy formula for repeated integration
in connection with similar proofs of Lemma \ref{Compactness}, \ref{WeakConvergence} and \ref{existence}, when $T$ is chosen to be sufficiently small. The latter, however, requires further analysis and modification of the proofs of Lemma \ref{Compactness} and Lemma \ref{WeakConvergence} based on the representation \eqref{represent}, which is beyond the scope of the current paper.
\end{remark}

\section*{Appendix}
The result below provides a compactness criterion for subsets of $L^2(\mu;\mathbb{R}^d)$ using the Malliavin calculus (see \cite{Prato}).
\begin{theorem}\label{the0}
Let $\{(\Omega,\mathcal{A},\mathbb{P});H\}$ be a Gaussian space, that is $(\Omega,\mathcal{A},\mathbb{P})$ is a probability space and $H$ a separable closed subspace of Gaussian random variables of $L^2(\Omega)$, which generate the $\sigma$-field $\mathcal{A}$. Denote by D the derivative operator acting on elementary smooth function variables in the sense that 
\begin{equation*}
 D(f(h_1,...,h_n))=\sum_{i=1}^n\delta_if(h_1,...,h_n)h_i,\quad h_i\in H, f\in\mathcal{C}_b^{\infty}(\mathbb{R}^n).
\end{equation*}
Further, let $D_{1,2}$ be the closure of the family of elementary smooth random variables with respect to the norm
\begin{equation*}
\big\vert\big\vert F \big\vert\big\vert_{1,2}:=\big\vert\big\vert F \big\vert\big\vert_{L^2(\Omega)}+\big\vert\big\vert DF \vert\vert_{L^2(\Omega;H)}.
\end{equation*}
Assume that $\mathcal{C}$ is a self-adjoint compact operator on $H$ with dense image. Then for any $c>0$, the set
\begin{equation*}
\mathcal{G}=\Big\{G\in D_{1,2}: \big\vert\big\vert G \big\vert\big\vert_{L^2(\Omega)}+\big\vert\big\vert \mathcal{C}^{-1} DG \big\vert\big\vert_{L^2(\Omega;H)}\le c \Big\}
\end{equation*}
is relatively compact in $L^2(\Omega)$.
\end{theorem}
In order to formulate compactness criteria useful for our purposes, we need the following technical result which also can be found in \cite{Prato}.
\begin{lemma}\label{the1}
Let $v_s$, $s\ge 0$, be the Haar basis of $L^2([0,1])$. For any $0<\alpha<1/2$ define the operator $A_{\alpha}$ on $L^2([0,1])$ by
\begin{equation*}
A_{\alpha}v_s=2^{k\alpha}v_s,\quad \text{if}\quad s=2^k+j
\end{equation*}
for $k\ge0$, $0\le j\le 2^k$ and
\begin{equation*}
A_{\alpha}1=1.
\end{equation*}
Then for $\beta$ with $\alpha<\beta<(1/2)$, there exists a constant $c_1$ such that
\begin{equation*}
\big\vert\big\vert A_{\alpha}f\big\vert\big\vert\le c_1\Bigg\{\big\vert\big\vert F \big\vert\big\vert_{L^2([0,1])}+\Bigg(\int_0^1\int_0^1\frac{\vert f(t)-f(t')\vert^2}{\vert t-t' \vert } dtdt'\Bigg)^{1/2}  \Bigg\}.
\end{equation*}
\end{lemma}
A direct consequence of Theorem \ref{the0} and Lemma \ref{the1} is the following compactness criteria:
\begin{corollary}
\label{MC}
 Let a sequence of $\mathcal{F}_1$-measurable random variables $X_n\in \mathbb{D}_{1,2}$, $n=1,2,...,$ be such that there exist constants $\alpha>0$ and $C>0$ with   
\begin{equation*}
\sup_n E\big[\big\vert\big\vert X_n \big\vert\big\vert^2\Big]\le C,
\end{equation*}
\begin{equation*}
\sup_n E\Big[\Big\vert\Big\vert D_tX_n-D_{t'}X_n \Big\vert\Big\vert^2\Big]\le C\vert t-t^{'} \vert^{\alpha},
\end{equation*}
for $0\le t'\le t\le 1$ and 
 \begin{equation*}
\sup_n\sup_{0\le t\le 1} E\Big[\big\vert\big\vert D_tX_n\big\vert\big\vert^2\Big]\le C.
\end{equation*}
Then the sequence $X_n$, $n=1,2,...$ is relatively compact in $L^2(\Omega)$.
\end{corollary}

To prove the main result (Theorem \ref{Main}), we also need the following proposition (see
Prop. 3.7 in \cite{MMNPZ}):

\begin{proposition}\label{prop1}
Let $f_{i}:\left[ 0,T\right] \times \mathbb{R}^{d}\longrightarrow \mathbb{R}%
,i=1,...,n$ be compactly supported smooth functions. Further, let $\alpha
_{i}=(a_{i,j})_{0\leq j\leq d}\in \left\{ 0,1\right\} ^{d}$ be a multiindex
with $\left\vert \alpha _{i}\right\vert :=\sum_{j=1}^{d}a_{i,j}=1$ for all $%
i=1,...,d$. Then there exists a constant $C$ depending on $d$ (but not on $n,
$ $f_{i},i=1,...,n$ and $\alpha _{i},i=1,...,d$) such that%
\begin{eqnarray}
&&\left\vert E\left[ \int_{t_{0}<t_{1}<...<t_{n}<t}\left(
\dprod\limits_{i=1}^{n}D^{\alpha _{i}}f_{i}(t_{i},B_{t_{i}})\right)
dt_{1}...dt_{n}\right] \right\vert   \label{Estimate} \\
&\leq &\frac{C^{n}\dprod \limits_{i=1}^{n}\left\Vert f_{i}\right\Vert
_{\infty }(t-t_{0})^{n/2}}{\Gamma (\frac{n}{2}+1)},  \notag
\end{eqnarray}
where $\Gamma $ is the Gamma function and where $D^{\alpha _{i}}$ denotes
the partial derivative with respect to the multiindex $\alpha _{i}$.
\end{proposition}

\end{document}